\documentclass[11pt]{amsart}%
\usepackage{amsmath}
\usepackage{graphicx}
\usepackage{amsfonts}
\usepackage{amssymb}%
\newtheorem{theorem}{Theorem}[section]
\theoremstyle{plain}

\newtheorem{lemma}{Lemma}[section]

\newtheorem{proposition}{Proposition}[section]

\numberwithin{equation}{section}
\begin{document}
\title[Gradient Estimates]{Explicit gradient estimates for minimal Lagrangian surfaces of dimension two}
\author{Micah Warren}
\author{Yu YUAN}
\address{Department of Mathematics, Box 354350\\
University of Washington\\
Seattle, WA 98195, USA}
\email{mwarren@math.washington.edu, yuan@math.washington.edu}
\thanks{Y.Y. is partially supported by an NSF grant.}
\date{August 1, 2007.}

\begin{abstract}
We derive explicit, uniform, a priori interior Hessian and gradient estimates
for special Lagrangian equations of all phases in dimension two.

\end{abstract}
\maketitle

\section{\bigskip Introduction}

In this note, we derive explicit \emph{interior a priori} Hessian and gradient
estimates for the special Lagrangian equation
\begin{equation}
\sum_{i=1}^{n}\arctan\lambda_{i}=\Theta\label{EsLag}%
\end{equation}
where $\lambda_{i}$ are the eigenvalues of the Hessian $D^{2}u$ and $n=2.$
\ \ Equation (\ref{EsLag}) stems from the special Lagrangian geometry [HL].
The Lagrangian graph $\left(  x,Du\left(  x\right)  \right)  \subset
\mathbb{R}^{n}\times\mathbb{R}^{n}$ is called special when the phase or the
argument of the complex number $\left(  1+\sqrt{-1}\lambda_{1}\right)
\cdots\left(  1+\sqrt{-1}\lambda_{n}\right)  $ is constant $\Theta,$ that is,
$u$ satisfies equation (\ref{EsLag}), and it is special if and only if
$\left(  x,Du\left(  x\right)  \right)  $ is a (volume minimizing) minimal
surface in $\mathbb{R}^{n}\times\mathbb{R}^{n}$ [HL, Theorem 2.3, Proposition
2.17]. Gradient estimates for the minimal Lagrangian surfaces are then Hessian
estimates for the special Lagrangian equation (\ref{EsLag}). When $n=2$, the
potential equation (\ref{EsLag}) also takes the equivalent form
\begin{equation}
\cos\Theta\bigtriangleup u+\sin\Theta\left(  \det D^{2}u-1\right)  =0.
\label{sigmaform2d}%
\end{equation}
\ \ We state our result in the following.

\begin{theorem}
Let $u$ be a smooth solution to (\ref{EsLag}) with $n=2$ on $B_{R}%
(0)\subset\mathbb{R}^{2}.$ Then the following both hold
\begin{equation}
\left\vert D^{2}u(0)\right\vert \leq C(2)\exp\left[  C(2)\max_{B_{R}%
(0)}|Du|^{2}/R^{2}\right]  ,\label{uniform}%
\end{equation}
or
\begin{equation}
\left\vert D^{2}u(0)\right\vert \leq C(2)\exp\left[  C(2)\frac{1}{\left\vert
\sin\Theta\right\vert ^{3/2}}\max_{B_{R}(0)}|Du|/R\right]  .\label{nonuniform}%
\end{equation}

\end{theorem}

In the 1950's, Heinz [H2] derived a Hessian bound for the two dimensional
Monge-Amp\`{e}re type equation including (\ref{sigmaform2d}). In the 1990's
Gregori [G] extended Heinz's estimate to a gradient bound in terms of the
heights of the two dimensional minimal surfaces with any codimension. A
gradient estimate for general dimensional and codimensional minimal graphs
with certain constraints on the gradients themselves was obtained in [W].

Although it is not clear whether the exponential dependence in our estimates
(\ref{uniform}) and (\ref{nonuniform}) is sharp, still it is sharper than the
double exponential dependence on $Du$ by Heinz [H2, Theorem 2], [H1, p.263,
p.255] and Gregori [G, Theorem 1], when applied to the special Lagrangian
equation of dimension two. On the other hand, like our nonuniform estimate
(\ref{nonuniform}), Heinz's estimate deteriorates as $\Theta$ goes to $0.$

In order to link the dependence of Hessian estimates in Theorem 1.1 to the
potential $u$ itself, we have the following.

\begin{theorem}
Let $u$ be a smooth solution to (\ref{EsLag}) on $B_{3R}(0)\subset
\mathbb{R}^{2}.$ Then we have
\begin{equation}
\max_{B_{R}(0)}|Du|\leq C\left(  2\right)  \left[  \operatorname*{osc}%
_{B_{3R}\left(  0\right)  }\frac{u}{R}+1\right]  . \label{gradient}%
\end{equation}

\end{theorem}

The strategies of our arguments are as follows. The function associated to the
volume element of the special Lagrangian graph, namely, $b=\ln$ $\sqrt
{\det\left(  I+D^{2}uD^{2}u\right)  },$ is subharmonic and satisfies a Jacobi
inequality. Using a Poincar\'{e} type inequality (instead of the usual mean
value inequality by Michael and Simon) together with the Jacobi inequality,
the maximum of $b$ on an interior region is bounded by the volume of the ball
on the minimal surface. Exploiting the divergence form of the volume element
of the minimal Lagrangian graphs, we bound the volume in terms of the height
of the special Lagrangian graph, which is the gradient of the solution to
equation (\ref{EsLag}). In order to push the resulting Hessian estimate
(\ref{nonuniform}) independent of the phase $\Theta,$ we first employ the Lewy
rotation technique to obtain a Hessian estimate for small phase $\Theta$ with
a constrained height, then combine it with (\ref{nonuniform}) to derive
(\ref{uniform}). Further, we obtain the uniform gradient estimate
(\ref{gradient}) independent of the phase $\Theta$ via the same Lewy rotation,
which links the corresponding estimates to the ones for harmonic functions.

More involved arguments are needed to obtain Hessian estimates for the special
Lagrangian equation (\ref{EsLag}) with $n=3$ and $\Theta\geq\pi/2$ in [WY2]
and [WY3]. The Bernstein-Pogorelov-Korevaar technique was employed to derive
Hessian estimates for (\ref{EsLag}) with certain constraints in [WY1]. The
problem of Hessian estimates for (\ref{EsLag}) with general phases $\Theta$
and general dimensions remain open to us.

\textbf{Notation. }$\partial_{i}=\frac{\partial}{\partial_{x_{i}}}%
,\ \partial_{ij}=\frac{\partial^{2}}{\partial x_{i}\partial x_{j}}%
,\ u_{i}=\partial_{i}u,\ u_{ji}=\partial_{ij}u$ etc., but $\lambda_{1}%
,\lambda_{2}$ do not represent the partial derivatives. Further, $h_{ijk}$
will denote (the second fundamental form)
\[
h_{ijk}=\frac{1}{\sqrt{1+\lambda_{i}^{2}}}\frac{1}{\sqrt{1+\lambda_{j}^{2}}%
}\frac{1}{\sqrt{1+\lambda_{k}^{2}}}u_{ijk}%
\]
when $D^{2}u$ is diagonalized. The constant $C(2)$ will denote various
dimensional constants, which do not depend on the phase $\Theta.$

\section{Preliminary inequalities}

Taking the gradient of both sides of the special Lagrangian equation
(\ref{EsLag}), we have
\begin{equation}
\sum_{i,j}^{n}g^{ij}\partial_{ij}\left(  x,Du\left(  x\right)  \right)  =0,
\label{Emin}%
\end{equation}
where $\left(  g^{ij}\right)  $ is the inverse of the induced metric
$g=\left(  g_{ij}\right)  =I+D^{2}uD^{2}u$ on the surface $\left(  x,Du\left(
x\right)  \right)  \subset\mathbb{R}^{n}\times\mathbb{R}^{n}.$ Simple
geometric manipulation of (\ref{Emin}) yields the usual form of the minimal
surface equation
\[
\bigtriangleup_{g}\left(  x,Du\left(  x\right)  \right)  =0,
\]
where the Laplace-Beltrami operator of the metric $g$ is given by
\[
\bigtriangleup_{g}=\frac{1}{\sqrt{\det g}}\sum_{i,j}^{n}\partial_{i}\left(
\sqrt{\det g}g^{ij}\partial_{j}\right)  .
\]
Because we are using harmonic coordinates $\bigtriangleup_{g}x=0,$ we see that
$\bigtriangleup_{g}$ also equals the linearized operator of the special
Lagrangian equation (\ref{EsLag}) at $u,$%
\[
\bigtriangleup_{g}=\sum_{i,j}^{n}g^{ij}\partial_{ij}.
\]
The gradient and inner product with respect to the metric $g$ are
\begin{align*}
\nabla_{g}v  &  =\left(  \sum_{k=1}^{n}g^{1k}v_{k},\cdots,\sum_{k=1}^{n}%
g^{nk}v_{k}\right)  ,\\
\left\langle \nabla_{g}v,\nabla_{g}w\right\rangle _{g}  &  =\sum_{i,j=1}%
^{n}g^{ij}v_{i}w_{j},\ \ \text{in particular \ }\left\vert \nabla
_{g}v\right\vert ^{2}=\left\langle \nabla_{g}v,\nabla_{g}v\right\rangle _{g}.
\end{align*}

We begin with some geometric inequalities.

\begin{lemma}
Let $u$ be a smooth solution to (\ref{EsLag}), with phase $\Theta\geq0,$ and
$n=2$. \ Set $b=\ln V=\ln\sqrt{\det\left(  I+D^{2}uD^{2}u\right)  }.$\ Then
$b$ satisfies%

\begin{equation}
\bigtriangleup_{g}b\;\geq\sin\Theta|\nabla_{g}b|^{2} \label{LapB1}%
\end{equation}
and for $\Theta\geq\pi/2,$%
\begin{equation}
\bigtriangleup_{g}b\;\geq|\nabla_{g}b|^{2}. \label{Jacobi lg phase}%
\end{equation}

\end{lemma}

\begin{proof}
[Proof]Assume that $D^{2}u$ is diagonalized at a point $p.$ \ The calculation
\begin{equation}
\bigtriangleup_{g}\ln\sqrt{(1+\lambda_{1}^{2})(1+\lambda_{2}^{2})}=\left[
4+\left(  \lambda_{1}+\lambda_{2}\right)  ^{2}\right]  \left(  h_{111}%
^{2}+h_{112}^{2}\right)  \label{Laplace log V}%
\end{equation}
follows from [Y1, Lemma 2.1], where we are using the notation $h_{ijk}%
=\sqrt{g^{ii}}\sqrt{g^{jj}}\sqrt{g^{kk}}u_{ijk}.$ \ Similarly,
\begin{align}
|\nabla_{g}b|^{2}  &  =\sum_{i=1}^{2}g^{ii}\left(  \partial_{i}\ln\sqrt{\det
g}\right)  ^{2}\nonumber\\
&  =\sum_{i=1}^{2}g^{ii}\left(  \frac{1}{2}\sum_{a,b=1}^{2}g^{ab}\partial
_{i}g_{ab}\right)  ^{2}=\sum_{i=1}^{2}g^{ii}\left(  \sum_{j=1}^{2}%
g^{jj}\lambda_{j}u_{jji}\right)  ^{2}\nonumber\\
&  =g^{11}\left[  g^{11}u_{111}\left(  \lambda_{1}-\lambda_{2}\right)
\right]  ^{2}+g^{22}\left[  g^{11}u_{112}\left(  \lambda_{1}-\lambda
_{2}\right)  \right]  ^{2}\nonumber\\
&  =\left(  h_{111}^{2}+h_{112}^{2}\right)  (\lambda_{1}-\lambda_{2})^{2},
\label{Gradient log V}%
\end{align}
where we used the minimal surface equation (\ref{Emin})%
\[
g^{11}u_{111}+g^{22}u_{221}=g^{11}u_{112}+g^{22}u_{222}=0.
\]
\ With (\ref{sigmaform2d}) in mind, we compute
\begin{align*}
4+\left(  \lambda_{1}+\lambda_{2}\right)  ^{2}-\sin\Theta(\lambda_{1}%
-\lambda_{2})^{2}  &  =4+\sigma_{1}^{2}-\sin\Theta\left(  \sigma_{1}%
^{2}-4\sigma_{2}\right) \\
&  =4+\sigma_{1}^{2}-\sin\Theta\left(  \sigma_{1}^{2}-4+4\cot\Theta\sigma
_{1}\right) \\
&  =(1-\sin\Theta)\left[  \sigma_{1}-2\frac{\cos\Theta}{(1-\sin\Theta
)}\right]  ^{2}\\
&  \geq0.
\end{align*}
Accordingly,
\[
\bigtriangleup_{g}b-\sin\Theta|\nabla_{g}b|^{2}=\left[  4+\left(  \lambda
_{1}+\lambda_{2}\right)  ^{2}-\sin\Theta(\lambda_{1}-\lambda_{2})^{2}\right]
\left(  h_{111}^{2}+h_{112}^{2}\right)  \geq0.
\]
The Jacobi inequality (\ref{LapB1}) is proved.\ 

For large phase, $\Theta\geq\pi/2,$ the equation (\ref{EsLag}) dictates that
both eigenvalues are positive, and one can see easily from
(\ref{Laplace log V}) and (\ref{Gradient log V}) that (\ref{Jacobi lg phase}) holds.
\end{proof}

In two dimensions, we take advantage of a certain \textquotedblleft super"
isoperimetric inequality on the level sets of \textquotedblleft
subharmonic\textquotedblright\ functions. The resulting Poincar\'{e} type
inequality can be used in place of the mean value inequality of Michael and
Simon in the proof of Theorem 1.1.

\begin{proposition}
Let $f$ be a smooth, positive function on $B_{2}(0)\subset\mathbb{R}^{2}.$
Suppose that $f$ satisfies the weak maximum principle: f attains its maximum
on the boundary of any subdomain of $B_{2}.$ \ Then
\[
||f||_{L^{\infty}(B_{1})}\leq\int_{B_{2}}\left\vert Df\right\vert
dx+\int_{B_{2}}fdx.
\]

\end{proposition}

\begin{proof}
Set $\alpha=\int_{B_{2}}fdx.$ We may assume $M\triangleq||f||_{L^{\infty
}(B_{1})}>\alpha.$ By Sard's theorem, the level set $\left\{  x|\ f\left(
x\right)  =t\right\}  \cap B_{2}$ is $C^{1}$ for almost all $t$ with
$\alpha\,<t<||f||_{L^{\infty}(B_{1})}.\ $For such (almost all) $t,$ we show
that $\left\{  x|\ f\left(  x\right)  =t\right\}  \cap B_{2}$ has length at
least $1$ in the following. The set $\left\{  x|\ f\left(  x\right)  \leq
t\right\}  \cap B_{1}$ is nonempty and satisfies%
\begin{equation}
\left\vert \left\{  x|\ f\left(  x\right)  \leq t\right\}  \cap B_{1}%
\right\vert >\left\vert B_{1}\right\vert -1,\label{fatset}%
\end{equation}
otherwise we have a contradiction:
\[
\alpha>\int_{B_{1}}fdx>t\ \left\vert \left\{  x|\ f\left(  x\right)
>t\right\}  \cap B_{1}\right\vert >\alpha.
\]
If any component of $\left\{  x|\ f\left(  x\right)  =t\right\}  \cap B_{2}$
stretches from the interior $B_{1}$ to the boundary $\partial B_{2},$ then the
length $\left\vert \left\{  x|\ f\left(  x\right)  =t\right\}  \cap
B_{2}\right\vert >1.$ Otherwise, each component of $\left\{  x|\ f\left(
x\right)  =t\right\}  \cap B_{2}$ which intersects \ $B_{1}$ must be a closed
curve in $B_{2},$ as we are using the fact that $t$ is not a critical value
for $f.$ From the maximum principle for $f,$ it follows that $f\leq t$ inside
any such closed curve. By (\ref{fatset}) and the usual isoperimetric
inequality for each of these (finitely many) \ $C^{1}$ regions where $f\leq
t,$ we have%
\begin{equation}
\left\vert \left\{  x|\ f\left(  x\right)  =t\right\}  \cap B_{2}\right\vert
\geq\sqrt{4\pi\left\vert \left\{  x|\ f\left(  x\right)  \leq t\right\}  \cap
B_{1}\right\vert }>1.\label{Length1}%
\end{equation}

Now we proceed as follows. For any $q\geq1,$%
\begin{align*}
\left[  \int_{B_{1}}\left\vert \left(  f-\alpha\right)  ^{+}\right\vert
^{q}dx\right]  ^{1/q} &  =\left[  \int_{0}^{M-\alpha}\left\vert \left\{
x|\ f\left(  x\right)  -\alpha>t\right\}  \cap B_{1}\right\vert dt^{q}\right]
^{1/q}\\
&  \leq\int_{0}^{M-\alpha}\left\vert \left\{  x|\ f\left(  x\right)
-\alpha>t\right\}  \cap B_{1}\right\vert ^{1/q}dt\\
&  \leq\left\vert B_{1}\right\vert ^{1/q}\int_{0}^{M-\alpha}\left\vert
\left\{  x|\ f\left(  x\right)  -\alpha=t\right\}  \cap B_{2}\right\vert dt\\
&  \leq\left\vert B_{1}\right\vert ^{1/q}\int_{B_{2}}\left\vert D\left[
f\left(  x\right)  -\alpha\right]  \right\vert dx,
\end{align*}
where the last inequality followed from the coarea formula; the second
inequality followed from (\ref{Length1}); and the first inequality followed
from the Hardy-Littlewood-Polya inequality for any nonnegative, nonincreasing
integrand $\eta\left(  t\right)  $ (cf. [BDM, p.258]):%
\[
\left[  \int_{0}^{T}\eta\left(  t\right)  ^{q}dt^{q}\right]  ^{1/q}\leq
\int_{0}^{T}\eta\left(  t\right)  dt.
\]
This H-L-P inequality is proved by noting that $s\eta\left(  s\right)
\leq\int_{0}^{s}\eta\left(  t\right)  dt$ and integrating the inequality%
\[
q\left[  s\eta\left(  s\right)  \right]  ^{q-1}\eta\left(  s\right)  \leq
q\left[  \int_{0}^{s}\eta\left(  t\right)  dt\right]  ^{q-1}\eta\left(
s\right)  =\frac{d}{ds}\left[  \int_{0}^{s}\eta\left(  t\right)  dt\right]
^{q}.
\]
Letting $q$ go to $\infty,$ we have%
\[
\left\Vert \left(  f-\alpha\right)  ^{+}\right\Vert _{L^{\infty}\left(
B_{1}\right)  }\leq\int_{B_{2}}\left\vert D\left(  f-\alpha\right)
\right\vert dx.
\]
Thus%
\[
||f||_{L^{\infty}(B_{1})}\leq\int_{B_{2}}\left\vert Df\right\vert
dx+\int_{B_{2}}fdx.
\]

\end{proof}

\section{Proof of Theorem 1.1}

We combine two estimates to obtain a uniform Hessian estimate for any given
height bound. The first estimate, which uses the Jacobi inequality,
deteriorates as $\Theta\rightarrow0.$ The second estimate holds for small
$\Theta$ with constrained height, and follows easily from a standard technique
for harmonic functions, combined with a Lewy rotation of coordinates. For
simplicity, we assume that $R=4$ and $u$ is a solution on $B_{4}%
\subset\mathbb{R}^{2}$. By scaling $u\left(  \frac{R}{4}x\right)  /\left(
\frac{R}{4}\right)  ^{2},$ we still get the estimate in Theorem 1.1.

\textbf{Case with }$\Theta$\textbf{-dependence.}\ By the symmetry of the
equation (\ref{EsLag}), we assume $\Theta>0.$ From inequality (\ref{LapB1}) in
Lemma 2.1, $b=\ln V$ is subharmonic with respect to the induced metric on
$B_{2};$ hence $b$ satisfies the weak maximum principle. We apply Proposition
2.1
\begin{align}
\left\Vert b\right\Vert _{L^{\infty}\left(  B_{1}\right)  } &  \leq\int
_{B_{2}}\left\vert Db\right\vert dx+\int_{B_{2}}b\;dx\nonumber\\
&  \leq\int_{B_{2}}\left\vert \nabla_{g}b\right\vert dv_{g}+\int_{B_{2}%
}b\;dx\nonumber\\
\leq &  \left(  \int_{B_{2}}\left\vert \nabla_{g}b\right\vert ^{2}Vdx\right)
^{1/2}\left(  \int_{B_{2}}Vdx\right)  ^{1/2}+\int_{B_{2}}Vdx.\label{Prob 21 b}%
\end{align}

Multiplying both sides of the Jacobi equation (\ref{LapB1}) by a non-negative
cut-off function $\psi\in C_{0}^{\infty}\left(  B_{3}\right)  $ with $\psi=1$
on $B_{2}$ and $\left\vert D\psi\right\vert \leq1.1,$ then integrating, we
obtain
\begin{align*}
\int_{B_{3}}\psi^{2}\left\vert \nabla_{g}b\right\vert ^{2}dv_{g}  &  \leq
\frac{1}{\sin\Theta}\int_{B_{3}}\psi^{2}\bigtriangleup_{g}b\;dv\\
&  =-\frac{1}{\sin\Theta}\int_{B_{3}}\left\langle 2\psi\nabla_{g}%
\varphi,\nabla_{g}b\right\rangle _{g}dv_{g}\\
&  \leq\frac{1}{2}\int_{B_{3}}\psi^{2}\left\vert \nabla_{g}b\right\vert
^{2}dv_{g}+2\left(  \frac{1}{\sin\Theta}\right)  ^{2}\int_{B_{3}}\left\vert
\nabla_{g}\psi\right\vert ^{2}dv_{g}.
\end{align*}
It follows that
\begin{equation}
\int_{B_{2}}\left\vert \nabla_{g}b\right\vert ^{2}Vdx\leq\int_{B_{3}}\psi
^{2}\left\vert \nabla_{g}b\right\vert ^{2}dv_{g}\leq4\csc^{2}\Theta\int
_{B_{3}}\left\vert \nabla_{g}\psi\right\vert ^{2}dv_{g}. \label{Int Jac}%
\end{equation}

Observe that by equation (\ref{EsLag}) or (\ref{sigmaform2d}) the volume
element takes a simple form%
\begin{align*}
V &  =\sqrt{\left(  1+\lambda_{1}^{2}\right)  \left(  1+\lambda_{2}%
^{2}\right)  }=\left\vert \left(  1+i\lambda_{1}\right)  \left(
1+i\lambda_{2}\right)  \right\vert =\left\vert 1-\sigma_{2}+i\sigma
_{1}\right\vert \\
&  =\frac{\sigma_{1}}{\sin\Theta}=\csc\Theta\bigtriangleup u.
\end{align*}
Hence,%
\[
\int_{B_{2}}Vdx\leq\frac{C(2)}{\sin\Theta}\left\Vert Du\right\Vert
_{L^{\infty}\left(  B_{2}\right)  }%
\]
and
\begin{align*}
\left\vert \nabla_{g}\psi\right\vert ^{2}V &  \leq\left(  \frac{\left\vert
D\psi\right\vert ^{2}}{1+\lambda_{1}^{2}}+\frac{\left\vert D\psi\right\vert
^{2}}{1+\lambda_{2}^{2}}\right)  V=\left\vert D\psi\right\vert ^{2}\left(
\frac{2+\lambda_{2}^{2}+\lambda_{1}^{2}}{V}\right)  \\
&  =\left\vert D\psi\right\vert ^{2}\left[  2(1-\sigma_{2})+\sigma_{1}%
^{2}\right]  \frac{\sin\Theta}{\sigma_{1}}=\left\vert D\psi\right\vert
^{2}\left(  2\cos\Theta+\sigma_{1}\sin\Theta\right)  ,
\end{align*}
where we used the equation (\ref{sigmaform2d}). We then have from
(\ref{Int Jac})
\begin{align*}
\int_{B_{2}}\left\vert \nabla_{g}b\right\vert ^{2}Vdx &  \leq C(2)\csc
^{2}\Theta\int_{B_{3}}\left(  2\cos\Theta+\sigma_{1}\sin\Theta\right)  dx\\
&  \leq\text{ }C(2)\left(  \csc^{2}\Theta+\csc\Theta\left\Vert Du\right\Vert
_{L^{\infty}\left(  B_{3}\right)  }\right)  .
\end{align*}
Thus from (\ref{Prob 21 b}),
\[
\left\Vert b\right\Vert _{L^{\infty}\left(  B_{1}\right)  }\leq C(2)\frac
{\left(  1+\sin\Theta\left\Vert Du\right\Vert _{L^{\infty}\left(
B_{3}\right)  }\right)  ^{\frac{1}{2}}}{\sin\Theta}\left[  \frac{\left\Vert
Du\right\Vert _{L^{\infty}\left(  B_{2}\right)  }}{\sin\Theta}\right]
^{\frac{1}{2}}+C\left(  2\right)  \frac{\left\Vert Du\right\Vert _{L^{\infty
}\left(  B_{2}\right)  }}{\sin\Theta};
\]
that is%

\begin{equation}
\left\Vert b\right\Vert _{L^{\infty}\left(  B_{1}\right)  }\leq C(2)\frac
{\left\Vert Du\right\Vert _{L^{\infty}\left(  B_{3}\right)  }}{\sin\Theta
}\left(  1+\frac{1}{\sin^{1/2}\Theta\left\Vert Du\right\Vert _{L^{\infty
}\left(  B_{3}\right)  }^{1/2}}\right)  . \label{Estimate 1}%
\end{equation}
The estimate (\ref{nonuniform}) follows by exponentiating (\ref{Estimate 1}).

Next, for very large phase, $\Theta>3\pi/4,$ we adapt the proof of
(\ref{Estimate 1}) to obtain a bound that does not deteriorate as
$\Theta\rightarrow\pi$ . First we note that from the Jacobi inequality
(\ref{Jacobi lg phase}) the $\Theta$-dependence in (\ref{Int Jac}) is no
longer needed, and we have \
\begin{equation}
\int_{B_{2}}|\nabla_{g}b|dv_{g}\leq4\left(  \int_{B_{3}}\left\vert \nabla
_{g}\psi\right\vert ^{2}dv_{g}\right)  ^{1/2}\left(  \int_{B_{2}%
(0)}Vdx\right)  ^{1/2}\leq C(2)\int_{B_{3}(0)}Vdx.\label{Pre Estimate 2}%
\end{equation}
Using another expression for the volume form
\[
V=\left\vert 1-\sigma_{2}+i\sigma_{1}\right\vert =\left\vert \sec
\Theta\right\vert \left(  \sigma_{2}-1\right)  \
\]
for $\Theta>\pi/2,$
\begin{align*}
\int_{B_{r}(0)}Vdx &  =\int_{B_{r}(0)}\left\vert \sec\Theta\right\vert \left(
\sigma_{2}-1\right)  dx\leq\left\vert \sec\Theta\right\vert \int_{B_{r}%
(0)}\sigma_{2}dx\\
&  =\left\vert \sec\Theta\right\vert \int_{B_{r}(0)}div(u_{1}u_{22,}%
-u_{1}u_{21})dx\\
&  \leq\left\vert \sec\Theta\right\vert \left\Vert Du\right\Vert _{L^{\infty
}\left(  B_{r}\right)  }\int_{\partial B_{r}(0)}|D^{2}u|ds.
\end{align*}
By the convexity of $u$ for large phase $\Theta>\pi/2,$ we know
$\bigtriangleup u\geq\left\vert D^{2}u\right\vert ,$ so
\[
\int_{B_{r}(0)}Vdx\leq\left\vert \sec\Theta\right\vert \left\Vert
Du\right\Vert _{L^{\infty}\left(  B_{r}\right)  }\int_{\partial B_{r}%
(0)}\bigtriangleup uds.
\]
Integrating the right hand side from $r=3$ to $r=4,$ we deduce
\begin{align*}
\int_{B_{3}(0)}Vdx &  \leq\left\vert \sec\Theta\right\vert \left\Vert
Du\right\Vert _{L^{\infty}\left(  B_{4}\right)  }\min_{r\in\lbrack3,4]}%
\int_{\partial B_{r}(0)}\bigtriangleup uds\\
&  \leq\left\vert \sec\Theta\right\vert \left\Vert Du\right\Vert _{L^{\infty
}\left(  B_{4}\right)  }\int_{B_{4}(0)}\bigtriangleup udx\leq\left\vert
\sec\Theta\right\vert \left\Vert Du\right\Vert _{L^{\infty}\left(
B_{4}\right)  }^{2}.
\end{align*}
In light of (\ref{Prob 21 b}) and (\ref{Pre Estimate 2}), we then have for
$\Theta>\pi/2,$%
\[
\left\Vert b\right\Vert _{L^{\infty}\left(  B_{1}\right)  }\leq C(2)\left\vert
\sec\Theta\right\vert \left\Vert Du\right\Vert _{L^{\infty}\left(
B_{4}\right)  }^{2},
\]
and finally%
\begin{equation}
\left\Vert D^{2}u\right\Vert _{L^{\infty}\left(  B_{1}\right)  }\leq
\exp\left[  C\left(  2\right)  \left\vert \sec\Theta\right\vert \left\Vert
Du\right\Vert _{L^{\infty}\left(  B_{4}\right)  }^{2}\right]
.\label{Estimate 2}%
\end{equation}
This finishes the estimates with $\Theta$-dependence in Theorem 1.1.

\textbf{Case without }$\Theta$\textbf{-dependence.} In order to prove the
Hessian bound (\ref{uniform}) that does not deteriorate for small $\Theta,$ we
need the following.

\begin{proposition}
Let $u$ be a smooth solution to (\ref{EsLag}) with $n=2$ and $\Theta\in\left[
0,\pi/4\right]  $ on $B_{1}(0)\subset\mathbb{R}^{2}.$\ Suppose that \
\[
\left\|  Du\right\|  _{L^{\infty}\left(  B_{1}\right)  }\leq\frac{1}%
{8\sin\Theta}.
\]
Then we have
\[
\left|  D^{2}u(0)\right|  \leq C(2)\left(  \left\|  Du\right\|  _{L^{\infty
}\left(  B_{1}\right)  }+1\right)  .
\]

\end{proposition}

\begin{proof}
\ We first find a harmonic representation of $\mathfrak{M}=\left(
x,Du\right)  $ via Lewy rotation (cf. [Y1], [Y2, p.1356]). We take a $U\left(
2\right)  $ rotation of $\mathbb{C}^{2}\cong\mathbb{R}^{2}\times\mathbb{R}%
^{2}:\bar{z}=e^{-\sqrt{-1}\Theta/2}z$ with $z=x+\sqrt{-1}y$ and $\bar{z}%
=\bar{x}+\sqrt{-1}\bar{y}.$ \ \ Because a $U\left(  2\right)  $ rotation
preserves the length and complex structure, $\mathfrak{M}$ is still a special
Lagrangian submanifold with the parametrization
\begin{equation}
\left\{
\begin{array}
[c]{c}%
\bar{x}=x\cos\frac{\Theta}{2}+Du\left(  x\right)  \sin\frac{\Theta}{2}\\
\bar{y}=-x\sin\frac{\Theta}{2}+Du\left(  x\right)  \cos\frac{\Theta}{2}%
\end{array}
\right.  . \label{harmonic}%
\end{equation}
In order to show that this parametrization is that of a gradient graph over
$\bar{x}$ , we show that $\bar{x}(x)$ is a diffeomorphism onto its image.
\ This is accomplished by showing that
\begin{equation}
\left\vert \bar{x}(x_{a})-\bar{x}(x_{b})\right\vert \geq\frac{1}{2\cos
\Theta/2}\left\vert x_{a}-x_{b}\right\vert \label{lbonrad}%
\end{equation}
for any $x_{a},$ $x_{b}$ . \ We assume by translation that $x_{b}=0$ and
$Du\left(  x_{b}\right)  =0.$ \ Now $\theta_{i}>\Theta-\frac{\pi}{2},$ so
$\ u+\frac{1}{2}\cot\Theta x^{2}$ is convex, \ and we have
\begin{gather*}
\left\vert \bar{x}\left(  x_{a}\right)  -\bar{x}\left(  0\right)  \right\vert
^{2}=\left\vert \bar{x}\left(  x_{a}\right)  \right\vert ^{2}=\left\vert
x_{a}\cos\frac{\Theta}{2}+Du\left(  x_{a}\right)  \sin\frac{\Theta}%
{2}\right\vert ^{2}\\
=\left\vert x_{a}\left(  \cos\frac{\Theta}{2}-\cot\Theta\sin\frac{\Theta}%
{2}\right)  +\left[  Du\left(  x_{a}\right)  +x_{a}\cot\Theta\right]
\sin\frac{\Theta}{2}\right\vert ^{2}\\
\geq\left\vert x_{a}\right\vert ^{2}\left(  \frac{\sin\frac{\Theta}{2}}%
{\sin\Theta}\right)  ^{2}+\left\vert Du\left(  x_{a}\right)  +x_{a}\cot
\Theta\right\vert ^{2}\sin^{2}\frac{\Theta}{2}+2\frac{\sin^{2}\frac{\Theta}%
{2}}{\sin\Theta}\left\langle x_{a},Du\left(  x_{a}\right)  +x_{a}\cot
\Theta\right\rangle \\
\geq|x_{a}|^{2}\left(  \frac{1}{2\cos\Theta/2}\right)  ^{2}.
\end{gather*}
We see that $\mathfrak{M}$ is a (special Lagrangian) graph over $\bar{x}$
space: $\mathfrak{M}=\left(  \bar{x},D\bar{u}\left(  \bar{x}\right)  \right)
,$ where $\bar{u}$ is a smooth function. \ Let $\bar{\lambda}_{i}$ be the
eigenvalues of the Hessian $D^{2}\bar{u}.$ Then
\begin{equation}
\arctan\bar{\lambda}_{i}=\arctan\lambda_{i}-\frac{\Theta}{2}\in\left(
-\frac{\pi}{2}+\frac{\Theta}{2}\,,\frac{\pi}{2}-\frac{\Theta}{2}\right)  .
\label{rotation range}%
\end{equation}
It follows that
\begin{align*}
\arctan\bar{\lambda}_{1}+\arctan\bar{\lambda}_{2}  &  =0\ \ \ \text{or}\\
\bigtriangleup\bar{u}  &  =0.
\end{align*}
Moreover, the domain of $\bar{u},$ $\bar{x}\left(  B_{1}\right)  $ contains a
ball in $\bar{x}$ space with radius $\bar{R}$ at least
\[
\bar{R}\geq\frac{1}{2\cos\Theta/2}%
\]
around $\bar{x}(0).$

For the harmonic function $\bar{u}_{\bar{e}\bar{e}}$ with $\bar{e}$ being an
arbitrary unit vector in $\bar{x}$ space, the mean value formula implies
\[
\bar{u}_{\bar{e}\bar{e}}\left(  \bar{0}\right)  =\frac{1}{\pi\bar{R}^{2}}%
\int_{\bar{B}_{\bar{R}}}\bar{u}_{\bar{e}\bar{e}}\ d\bar{x}\leq\frac{2\pi
\bar{R}}{\pi\bar{R}^{2}}\left\Vert D\bar{u}\right\Vert _{L^{\infty}\left(
\bar{B}_{\bar{R}}\right)  }\leq4\cos\frac{\Theta}{2}\left\Vert D\bar
{u}\right\Vert _{L^{\infty}\left(  \bar{B}_{\bar{R}}\right)  }.
\]
From the above harmonic parametrization (\ref{harmonic}) of $\mathfrak{M},$ we
know
\[
\left\Vert D\bar{u}\right\Vert _{L^{\infty}\left(  \bar{B}_{\bar{R}}\right)
}\leq\sin\frac{\Theta}{2}+\cos\frac{\Theta}{2}\left\Vert Du\right\Vert
_{L^{\infty}\left(  B_{1}\right)  }.
\]
Thus we get
\[
\bar{\lambda}_{i}\left(  \bar{0}\right)  \leq4\cos\frac{\Theta}{2}\left[
\sin\frac{\Theta}{2}+\cos\frac{\Theta}{2}\left\Vert Du\right\Vert _{L^{\infty
}\left(  B_{1}\right)  }\right]  .
\]
From (\ref{rotation range}), we see that
\[
\lambda_{i}\left(  0\right)  =\frac{\bar{\lambda}_{i}\left(  \bar{0}\right)
+\tan\frac{\Theta}{2}}{1-\bar{\lambda}_{i}\left(  \bar{0}\right)  \tan
\frac{\Theta}{2}}.
\]
Note that $\lambda_{\max}\left(  0\right)  \geq\left\vert \lambda_{i}\left(
0\right)  \right\vert $ for $\Theta\geq0.$ It follows that
\begin{align*}
\left\vert D^{2}u\left(  0\right)  \right\vert  &  \leq\lambda_{\max}\left(
0\right) \\
&  \leq\frac{\tan\frac{\Theta}{2}+4\cos\frac{\Theta}{2}\left[  \sin
\frac{\Theta}{2}+\cos\frac{\Theta}{2}\left\Vert Du\right\Vert _{L^{\infty
}\left(  B_{1}\right)  }\right]  }{2\cos\Theta-1-2\sin\Theta\left\Vert
Du\right\Vert _{L^{\infty}\left(  B_{1}\right)  }}\\
&  \leq C\left(  2\right)  \left[  \left\Vert Du\right\Vert _{L^{\infty
}\left(  B_{1}\right)  }+1\right]  ,
\end{align*}
provided that, say, $\left[  2\cos\Theta-1-2\sin\Theta\left\Vert Du\right\Vert
_{L^{\infty}\left(  B_{1}\right)  }\right]  >0.15,$ which is available under
our assumption.
\end{proof}

We finish the proof of Theorem 1.1 without $\Theta$-dependence. By symmetry,
we only consider the cases $0\leq\Theta<\pi.$ \ \ 

We first consider the small phase $0\leq\Theta\leq\pi/4.$

If \ $\left\Vert Du\right\Vert _{L^{\infty}\left(  B_{3}\right)  }\geq\frac
{1}{8\sin\Theta},$ then from the first estimate (\ref{Estimate 1}) we have
\begin{align*}
\left\Vert D^{2}u\right\Vert _{L^{\infty}\left(  B_{1}\right)  }  &  \leq
C(2)\exp\left\{  C(2)\frac{\left\Vert Du\right\Vert _{L^{\infty}\left(
B_{3}\right)  }}{\sin\Theta}\left[  1+\frac{1}{\sin^{1/2}\Theta\left\Vert
Du\right\Vert _{L^{\infty}\left(  B_{3}\right)  }^{1/2}}\right]  \right\} \\
&  \leq C(2)\exp\left[  C(2)\left\Vert Du\right\Vert _{L^{\infty}\left(
B_{4}\right)  }^{2}\right]  .
\end{align*}

If $\left\Vert Du\right\Vert _{L^{\infty}\left(  B_{2}\right)  }<\frac
{1}{8\sin\Theta},$ then Proposition 3.1 (applied at any point in $B_{1})$
implies
\begin{align*}
\left\Vert D^{2}u\right\Vert _{L^{\infty}\left(  B_{1}\right)  }  &  \leq
C\left(  2\right)  \left[  \left\Vert Du\right\Vert _{L^{\infty}\left(
B_{2}\right)  }+1\right] \\
&  \leq C(2)\exp\left[  C(2)\left\Vert Du\right\Vert _{L^{\infty}\left(
B_{4}\right)  }^{2}\right]  .
\end{align*}

For phases $\pi/4\leq\Theta\leq3\pi/4,$ \ $\sin\Theta$ is bounded away from
$0$ and (\ref{Estimate 1}) gives
\begin{align*}
\left\Vert D^{2}u\right\Vert _{L^{\infty}\left(  B_{1}\right)  }  &  \leq
C(2)\exp\left(  C(2)\left\Vert Du\right\Vert _{L^{\infty}\left(  B_{3}\right)
}\right) \\
&  \leq C(2)\exp\left(  C(2)\left\Vert Du\right\Vert _{L^{\infty}\left(
B_{4}\right)  }^{2}\right)  .
\end{align*}

For large phase $\Theta\geq3\pi/4,$ $\sec\Theta$ is bounded and we have from
(\ref{Estimate 2})
\[
\left\Vert D^{2}u\right\Vert _{L^{\infty}\left(  B_{1}\right)  }\leq
C(2)\exp\left(  C(2)\left\Vert Du\right\Vert _{L^{\infty}\left(  B_{4}\right)
}^{2}\right)  .
\]
The proof of estimate (\ref{uniform}) without $\Theta$-dependence in Theorem
1.1 is complete after a scaling.

\section{Proof of Theorem 1.2}

By symmetry we assume $\Theta\geq0.$ By scaling $u\left(  \frac{R}{3}x\right)
/\left(  \frac{R}{3}\right)  ^{2},$ we may assume that $u$ is a solution on
$B_{3}(0).$\ If $\Theta=0,$ then $u$ is harmonic and the linear gradient
estimate is standard. Otherwise, using
\[
\arctan\lambda_{i}>\Theta-\frac{\pi}{2}%
\]
we can control the gradient of the convex function $u\left(  x\right)
+\frac{1}{2}\max\left\{  \cot\Theta,0\right\}  x^{2}$ by its oscillation.
Thus
\begin{equation}
\left\vert Du\left(  0\right)  \right\vert \leq\operatorname*{osc}_{B_{1}%
}u+\frac{1}{2}\max\left\{  \cot\Theta,0\right\}  . \label{Rough Bound}%
\end{equation}
The following uses the same rotation argument as in Theorem 1.1 to deal with
very small $\Theta.$

\begin{proposition}
Let $u$ satisfy (\ref{EsLag}) with $\Theta\in\left(  0,\pi/4\right)  $ on
$B_{2}\left(  0\right)  .$ Suppose that%

\begin{equation}
\operatorname*{osc}_{B_{2}}u\leq\frac{1}{2\sin\Theta}. \label{osc condition}%
\end{equation}
Then
\[
\left\vert Du\left(  0\right)  \right\vert \leq C\left(  2\right)  \left(
\operatorname*{osc}_{B_{2}}u+1\right)  .
\]

\end{proposition}

\begin{proof}
We perform a Lewy rotation as before, to obtain a harmonic representation
$\mathfrak{M}=\left(  \bar{x},D\bar{u}\left(  \bar{x}\right)  \right)  $ for
the original special Lagrangian graph $\mathfrak{M}=(x,Du(x))$ with $x\in
B_{2}.$\ Recentering the new coordinates, we take
\begin{equation}
\left\{
\begin{array}
[c]{l}%
\bar{x}=x\cos\frac{\Theta}{2}+Du\left(  x\right)  \sin\frac{\Theta}%
{2}-Du\left(  0\right)  \sin\frac{\Theta}{2}\\
D\bar{u}\left(  \bar{x}\right)  =-x\sin\frac{\Theta}{2}+Du\left(  x\right)
\cos\frac{\Theta}{2}%
\end{array}
\right.  \label{harmonic'}%
\end{equation}
By (\ref{lbonrad}) we see that the harmonic function $\bar{u}$ is defined on a
ball of radius
\[
\bar{R}=\frac{2}{2\cos(\frac{\Theta}{2})}>1
\]
in $\bar{x}$-space around $\bar{0}.$

From (\ref{harmonic'}) and the classical estimate on the derivative of the
harmonic function $\bar{u},$ we have%
\[
\left\vert Du(0)\right\vert =\frac{\left\vert D\bar{u}(\bar{0})\right\vert
}{\cos(\Theta/2)}\leq C\left(  2\right)  \max_{\bar{B}_{1}\left(  \bar
{0}\right)  }\left\vert \bar{u}-\bar{u}\left(  \bar{0}\right)  \right\vert .
\]

We may assume that $\bar{u}(\bar{0})=0.$ The maximum of $|\bar{u}|$ on
$\bar{B}_{1}(\bar{0})$ must occur on the boundary, without loss of generality
we assume this happens along the positive $\bar{x}_{1}$-axis. Thus we have
\[
\max_{\bar{B}_{1}\left(  \bar{0}\right)  }\left\vert \bar{u}\right\vert
=\left\vert \int_{\bar{x}_{1}=0}^{\bar{x}_{1}=1}\bar{u}_{\bar{x}_{1}}d\bar
{x}_{1}\right\vert .
\]
In the following, we convert the integral of $\bar{u}_{\bar{x}_{1}}$ to one in
terms of $u_{x_{1}},$ then recover the oscillation of $\bar{u}$ from that of
$u.$

We work on the $x_{1}$-$y_{1}$ plane in the remaining proof. Under our above
assumption, the $\bar{x}_{1}$-axis is given by the line
\[
y_{1}=\tan\left(  \frac{\Theta}{2}\right)  x_{1}%
\]
and the curve $\gamma:(x_{1},u_{1}(x_{1}))$ with $\left\vert x_{1}\right\vert
\,<2$ forms a graph over $\bar{x}_{1}$-axis. \ Let $l_{0}$ be the line
perpendicular to $\bar{x}_{1}$-axis and intersecting the curve $\gamma$ at
$\left(  0,u_{1}\left(  0\right)  \right)  $ along the $y_{1}$-axis. The
intersection of $l_{0}$ and the $\bar{x}_{1}$-axis (which is also the origin
of the recentered $\bar{x}_{1}$-$\bar{y}_{1}$ plane) has distance to the
origin of $x_{1}$-$y_{1}$ plane given by
\begin{equation}
\left\vert u_{1}\left(  0\right)  \right\vert \sin\left(  \frac{\Theta}%
{2}\right)  \leq\left(  \operatorname*{osc}_{B_{1}}u+\frac{1}{2}\cot
\Theta\right)  \sin\left(  \frac{\Theta}{2}\right)  \leq1\label{cond41}%
\end{equation}
by the rough bound (\ref{Rough Bound}) and the condition (\ref{osc condition}%
). Now let $l_{1}$ be the line parallel to $l_{0}$ passing through the point
$\bar{x}_{1}=1$ along the $\bar{x}_{1}$-axis.

\ The integral
\[
\int_{\bar{x}_{1}=0}^{\bar{x}_{1}=1}\bar{u}_{\bar{x}_{1}}d\bar{x}_{1}%
\]
is the signed area between the $\bar{x}_{1}$-axis and the curve $\gamma,$ and
lying between the lines $l_{0}$ and $l_{1}.$ We convert this to an integral
over $x_{1},$
\[
\int_{\bar{x}_{1}=0}^{\bar{x}_{1}=1}\bar{u}_{\bar{x}_{1}}d\bar{x}_{1}%
=\int_{P(l_{0}\cap\bar{x}_{1}\text{-axis})}^{P(l_{1}\cap\bar{x}_{1}%
\text{-axis})}\left[  u_{1}(x_{1})-\tan\left(  \frac{\Theta}{2}\right)
x_{1}\right]  dx_{1}+K_{0}+K_{1},
\]
where $P$ denotes projection to the $x_{1}$-axis, and $K_{0}\ $as well as
$K_{1}$ denotes the signed areas to the left or right of the desired region,
forming the difference.

It is important to note the following for $j=1,2:$

(i)$\ P(l_{j}\cap\bar{x}_{1}$-axis$)\ $is in the $x_{1}$-domain of $u_{1}$ by
(\ref{cond41}),%
\begin{align*}
\left\vert P(l_{0}\cap\bar{x}_{1}\text{-axis})\right\vert  &  \leq1\cdot
\cos\left(  \frac{\Theta}{2}\right)  <1,\\
\left\vert P(l_{1}\cap\bar{x}_{1}\text{-axis})\right\vert  &  \leq\left(
1+1\right)  \cdot\cos\left(  \frac{\Theta}{2}\right)  <2;
\end{align*}

(ii) $P(l_{j}\cap\gamma)\ $is also in the $x_{1}$-domain of $u_{1}$ as
$\gamma$ is a graph over $B_{2},$%
\[
\left\vert P(l_{j}\cap\gamma)\right\vert \leq2;
\]

(iii) the region $K_{j}$ is bounded by the line $l_{j},$ the vertical line
$x_{1}=P(l_{j}\cap\bar{x}_{1}$-axis$),$ and the curve $\gamma,$ also each
region $K_{j}$ is on one side of the $\bar{x}_{1}$-axis.

Thus from (i)
\[
\left\vert \int_{P(l_{0}\cap\bar{x}_{1}\text{-axis})}^{P(l_{1}\cap\bar{x}%
_{1}\text{-axis})}\left[  u_{1}(x_{1})-\tan\left(  \frac{\Theta}{2}\right)
x_{1}\right]  dx_{1}\right\vert \leq\operatorname*{osc}_{B_{2}}u+C(2)
\]
and from (ii) (iii)%
\[
|K_{j}|\leq\left\vert \int_{P(l_{j}\cap\bar{x}_{1}\text{-axis})}^{P\left[
l_{j}\cap\gamma\right]  }\left[  u_{1}(x_{1})-\tan(\Theta/2)x_{1}\right]
dx_{1}\right\vert \leq\operatorname*{osc}_{B_{2}}u+C(2).
\]
It follows that
\[
\left\vert Du(0)\right\vert \leq C(2)\max_{\bar{B}_{1}\left(  \bar{0}\right)
}\left\vert \bar{u}-\bar{u}\left(  \bar{0}\right)  \right\vert \leq C\left(
2\right)  \left(  \operatorname*{osc}_{B_{2}}u+1\right)  .
\]

\end{proof}

\bigskip

We complete the proof of Theorem 1.2. \ For $\Theta\geq\pi/4,$ the bound
(\ref{Rough Bound}) gives
\[
\left\vert Du\left(  0\right)  \right\vert \leq\operatorname*{osc}_{B_{1}%
}u+\frac{1}{2}\leq C(2)\left[  \operatorname*{osc}_{B_{2}}u+1\right]  .
\]
For $\Theta\leq\pi/4,$ if $\operatorname*{osc}_{B_{2}}u\leq1/\left(
2\sin\Theta\right)  ,$ then Proposition 4.1 gives
\[
\left\vert Du(0)\right\vert \leq C(2)\left[  \operatorname*{osc}_{B_{2}%
}u+1\right]  .
\]
Otherwise, $\ \operatorname*{osc}_{B_{2}}u>1/\left(  2\sin\Theta\right)  ,$
and from (\ref{Rough Bound})
\[
\left\vert Du\left(  0\right)  \right\vert \leq\operatorname*{osc}_{B_{1}%
}u+\operatorname*{osc}_{B_{2}}u\leq C(2)\left[  \operatorname*{osc}_{B_{2}%
}u+1\right]  .
\]
Applying this estimate on $B_{2}(x)$ for any $x\in B_{1}(0),$ we arrive at the
conclusion of Theorem 1.2.

\end{document}